\documentclass[11pt,reqno]{amsart}
\usepackage{amssymb}
\usepackage{amsfonts}
\usepackage{amsthm}
\usepackage{amscd}
\numberwithin{equation}{section}
\theoremstyle{plain}
\newtheorem{theorem}{Theorem}[section]
\newtheorem{proposition}[theorem]{Proposition}
\newtheorem{corollary}[theorem]{Corollary}

\theoremstyle{definition}
\newtheorem{definition}[theorem]{Definition}
\newtheorem{remark}[theorem]{Remark}
\newtheorem*{example}{Example}
\newcommand{\refE}[1]{(\ref{eq:#1})}

\newcommand{\refS}[1]{Section~\ref{sec:#1}}
\newcommand{\refT}[1]{Theorem~\ref{T:#1}}

\newcommand{\refP}[1]{Proposition~\ref{P:#1}}
\newcommand{\refD}[1]{Definition~\ref{D:#1}}

\newcommand{\C}{\ensuremath{\mathbb{C}}}

\newcommand{\Z}{\ensuremath{\mathbb{Z}}}
\newcommand{\K}{\ensuremath{\mathcal{K}}}

\newcommand{\J}{\ensuremath{\mathbb{J}}}
\newcommand{\Jl}{\ensuremath{\mathbb{J}_\lambda}}
\newcommand{\Pro}{\ensuremath{\mathbb{P}}}

\newcommand{\cins}{\frac 1{2\pi\mathrm{i}}\int_{C_S}}

\newcommand{\cintl}[1]{\frac 1{24\pi\mathrm{i}}\int_{#1 }}

\newcommand{\A}{\mathcal{A}}

\newcommand{\La}{\mathcal{L}}
\newcommand{\Sp}{\mathcal{F}^{-1/2}}
\newcommand{\Sa}{\mathcal{S}}

\newcommand{\ord}{\operatorname{ord}}
\newcommand{\res}{\operatorname{res}}
\newcommand{\fpz}{\frac {d }{dz}}
\newcommand{\pfz}[1]{\frac {d#1}{dz}}

\newcommand{\Ho}{\mathrm{H}}

\newcommand{\al}{\ensuremath{\alpha}}
\newcommand{\be}{\ensuremath{\beta}}

\newcommand{\Fl}[1][\lambda]{\mathcal{F}^{#1}}

\newcommand{\Do}{\mathcal{D}^1}

\newcommand{\kndual}[2]{\langle #1,#2\rangle}
\newcommand{\ldot}{\,.\,}

\newcommand{\la}{\lambda}

\newcommand{\sbul}{{\boldsymbol{\cdot}}}

\begin{document}
\title[Lie superalgebras of Krichever-Novikov type]
{Lie superalgebras of Krichever-Novikov type and their 
central extensions}

\author[Martin Schlichenmaier]{Martin Schlichenmaier}
\thanks{Partial  support by
the 
Internal Research Project  GEOMQ11,  University of Luxembourg,
is acknowledged.}
\address{%
University of Luxembourg\\
Mathematics Research Unit, FSTC\\
Campus Kirchberg\\ 6, rue Coudenhove-Kalergi,
L-1359 Luxembourg-Kirchberg\\ Luxembourg
}
\email{martin.schlichenmaier@uni.lu}
\begin{abstract}
Classically
important examples of
Lie superalgebras 
have been constructed starting from the Witt and Virasoro 
algebra. In this article 
we consider Lie superalgebras of 
 Krichever-Novikov type.
These algebras are multi-point and higher genus equivalents. 
The grading in the classical case is replaced by an almost-grading.
The almost-grading is
determined by a splitting of the set of points were poles are allowed into
two disjoint subsets. 
With respect to  a fixed splitting, 
or equivalently with respect to an 
 almost-grading, it is shown
that there is up to rescaling and equivalence a unique 
non-trivial central extension. It is given explicitly.
Furthermore, a complete classification of bounded cocycles
(with respect to the almost-grading) is given.
\end{abstract}
\date{2.1.2013, minor revisions 20.3.2013}
\maketitle

\vskip 1.0cm
\section{Introduction}\label{S:intro}
Krichever--Novikov (KN) type algebras give important examples of infinite
dimensional algebras. They are defined 
via meromorphic objects on compact Riemann surfaces $\Sigma$ of
arbitrary genus with controlled polar behaviour. More precisely,
poles are only allowed at a fixed finite set of points denoted
by $A$. The classical examples are the algebras defined by objects
on the Riemann sphere (genus zero) with possible poles only at
$\{0,\infty\}$. This yields e.g. the well-known Witt algebra, current
algebras,
and their central extensions the Virasoro, and the affine Kac-Moody
algebras.
For higher genus, but still only for two points were poles are
allowed, they were generalised by Krichever and Novikov 
\cite{KNFa}, \cite{KNFb}, \cite{KNFc} in 1986.
In 1990 the author 
\cite{SLa}, \cite{SLb}, \cite{SLc}, \cite{SDiss} extended the approach
further to the general multi-point case. 
This extension was not a straight-forward generalization.
The crucial point was to introduce a replacement of the
graded algebra structure present in the ``classical'' case.
Krichever and Novikov found that an almost-grading, see
\refD{almgrad}, will be enough to allow constructions in
representation theory, like triangular decomposition, highest
weight modules, Verma modules and so on.
In \cite{SLc}, \cite{SDiss} it was realized that a splitting
of $A$ two disjoint non-empty subsets $A=I\cup O$ is crucial for
introducing an almost-grading
and the corresponding almost-grading was given.
In the classical situation there is only one such splitting
(up to inversion) hence there is only one almost-grading, which is
indeed a grading.
Similar to  the classical situation, a Krichever-Novikov algebra
should
always be considered as an algebra of meromorphic objects
with an almost-grading coming from such a splitting.

I like to point out that already in the genus zero case 
(i.e. the Riemann sphere case) with more
than two points where  poles are allowed the algebras will be only 
almost-graded. In fact, quite a number of interesting
new phenomena will show up already there, see
\cite{SchlDeg}, \cite{FiaSchl1}, \cite{FiaSchlaff}.

\medskip
In the context of conformal field theory and string theory
superextensions of the classical algebras appeared, see e.g.
\cite{GSW1}. Very important examples are the Neveu-Schwarz and the
Ramond type superalgebras.
Quite soon some physicists also studied superanalogs of the
algebra of Krichever-Novikov type, but still only with two points were
poles are allowed, e.g.
\cite{BMRR}, \cite{BLMR}, \cite{BMTW}, \cite{Bryant}, \cite{Zac}.
The multi-point case was also developed by the author. It 
has not been published yet, but see \cite{Schlknbook}.

Quite recently, these superalgebras of Krichever--Novikov type found
again
interest in the context of Jordan superalgebras and Lie antialgebras
(see Ovsienko \cite{Ovs1} for their definitions and Lecomte and
Ovsienko\cite{LecOvs} for further properties).
Starting from  Krichever--Novikov type superalgebras  interesting
explicite infinite dimensional examples of Jordan superalgebras 
and antialgebras
can be constructed.
In this respect, see the work of Leidwanger and Morier-Genoud 
\cite{LeiMor}, \cite{LeiMor1}, and Kreusch \cite{Kreusch}.

\medskip

The goal of this article is to recall the general definition of 
KN algebras for the multi-point situation and for arbitrary 
genus. The classical situation will be a special case.
In particular, the construction of the Lie superalgebra is
recalled. Its almost-graded structure, induced by a fixed
splitting $A=I\cup O$ is given.
Also the Jordan superalgebra of KN type 
(with its almost-grading) fits perfectly in this picture.
See Remark \ref{jordan}.

One of the main results of the paper is the proof that there
is, up to rescaling the central element and equivalence, only 
one non-trivial almost-graded central extension  of the Lie superalgebra
of KN type with even central element.
We stress the fact, that this does not mean that there is
essentially only one central extension.
In fact, a different splitting of $A$ will yield a different
almost-grading and hence an essentially different 
central extension. Moreover, at least for higher genus, 
there are central extensions which are not related to any
almost-grading.
In the classical situation we reprove uniqueness of the
non-trivial central extension.

We will give a geometric description for the defining cocycle,
see \refE{sacoc}.
For the two-point case the form of cocycle was given 
by Bryant in \cite{Bryant},
correcting some ommission in \cite{BMRR}.

A cocycle is bounded from above if 
its value is zero
if the sum of the
degrees of the (homogeneous) arguments are higher than a certain
bound. Krichever and Novikov introduced the term ``local'' cocycle,
for a cocycle which is bounded from above and from below.
Local cocycles are exactly those cocycles which define
central extensions which allow that the almost-grading can 
be extended to them.
In the process of proving the uniqueness of local cocycle classes
(\refT{main})
we give a complete classification of bounded (from above) cocycles.
To prove this we show
the fact that  a bounded cocycle 
of the Lie superalgebra  is already fixed by its restriction 
to the vector field subalgebra
(\refP{vanish}). For the vector field algebra
the bounded cocycle were classified by the author \cite{Scocyc}.

Up to this point we assumed that the central element was an even
element. In an additional section we consider the case that
the central element is odd. We show that all bounded (from
above) cocycles for
odd central elements are cohomologically trivial. This means that
the corresponding central extension of the Lie superalgebra will
split.

We close with some remarks on special examples.

\section{The algebras}\label{sec:algebra}
\subsection{The geometric set-up}
For the whole article let $\Sigma$ be a compact Riemann surface 
without any restriction for the genus $g=g(\Sigma)$.
Furthermore, let $A$ be a finite subset of $\Sigma$.
Later we will need a splitting of $A$ into two non-empty disjoint
subsets $I$ and $O$, i.e. $A=I\cup O$. Set $N:=\#A$,
$K:=\#I$, $M:=\#O$, with $N=K+M$. 
More precisely, let
\begin{equation}
I=(P_1,\ldots,P_K),\quad\text{and}\quad
O=(Q_1,\ldots,Q_{M})
\end{equation}
be disjoint  ordered tuples of  distinct points (``marked points'',
``punctures'') on the Riemann surface.
In particular, we assume $P_i\ne Q_j$ for every
pair $(i,j)$. The points in $I$ are
called the {\it in-points}, the points in $O$ the {\it out-points}.
Sometimes we consider $I$ and $O$ simply as sets.

In the article we sometimes refer to the classical situation. By this
we understand $\Sigma=S^2$, the Riemann sphere, or equivalently the
projective line over $\C$, $I=\{0\}$ and $O=\{\infty\}$ with respect
to the quasi-global coordinate $z$.

Our objects, algebras, structures, ...  will be
meromorphic objects 
defined on $\Sigma$ which are holomorphic outside  the points in
$A$.
To introduce them 
let $\K=\K_\Sigma$ be the canonical line bundle of $\Sigma$,
resp. the locally free canonically sheaf.
The local sections of the bundle are
the local holomorphic differentials.
If $P\in\Sigma$ is a point and $z$ a local holomorphic coordinate
at $P$ then a local holomorphic differential 
can be written as $f(z)dz$ with a local holomorphic function 
$f$ defined in a neighbourhood of $P$.
A global holomorphic section can be described locally for a covering
by coordinate charts $(U_i,z_i)_{i\in J}$ by a system of
local holomorphic functions $(f_i)_{i\in J}$, which are related by 
the transformation rule induced by the
coordinate change map $z_j=z_j(z_i)$ and the condition
$f_idz_i=f_jdz_j$
\begin{equation}\label{eq:ktrans}
f_j=f_i\cdot \left(\frac {dz_j}{dz_i}\right)^{-1}.
\end{equation}
With 
respect to a coordinate covering
a meromorphic section of $\K$ is given as 
a collection of local meromorphic functions $(h_i)_{i\in J}$ 
for which the transformation
law \refE{ktrans} is true.

In the following $\la$ is either an integer or a half-integer.
If $\la$ is an integer then 
\newline
(1) $\K^{\la}=\K^{\otimes \la}$ 
for $\la>0$, 
\newline
(2)  $\K^{0}=\mathcal{O}$,
 the trivial line bundle, and 
\newline
(3) $\K^{\la}=(\K^*)^{\otimes (-\la)}$ 
for $\la<0$.
\newline
Here as usual $\K^*$ denotes the  dual line bundle
to the canonical line bundle.
The dual line bundle is the holomorphic tangent line bundle, whose
local sections are the holomorphic tangent vector fields
$f(z)(d/d z)$.
If $\la$ is a half-integer, then we first have to fix a ``square
root''
of the canonical line bundle, sometimes called a
\emph{theta-characteristics}.
This means we fix a line bundle $L$ for
which $L^{\otimes 2}=K$.

After such a choice of $L$ is done we set 
$\K^{\la}=\K^{\la}_L=L^{\otimes 2\la}$.
In most cases we will drop  mentioning $L$, but we have to
keep the choice in mind.
Also the structure of the 
algebras we are about to define will depend on the
choice. But the main  properties will remain the same.

\begin{remark}
A Riemann surface of genus $g$ has exactly 
$2^{2g}$ non-isomorphic square roots of $\K$. 
For $g=0$ we have $\K=\mathcal{O}(-2)$ and 
$L=\mathcal{O}(-1)$, the tautological bundle which 
 is the unique square root.
Already for $g=1$ we have 4 non-isomorphic ones.
As in this case $\K=\mathcal{O}$ one solution is 
$L_0=\mathcal{O}$. 
But we have also other bundles $L_i$, $i=1,2,3$.
Note that $L_0$ has a non-vanishing global holomorphic
section, whereas  $L_1,L_2$, $L_3$ do not have a global holomorphic
section. 
In general, 
depending on the parity of $\dim\Ho(\Sigma,L)$, one 
distinguishes even and odd theta characteristics $L$. For $g=1$ 
the bundle $\mathcal{O}$ is  odd, the others are even
theta characteristics.
\end{remark}

\bigskip
We set
\begin{multline}\qquad
\Fl:=\Fl(A):=\{f \text{ is a global meromorphic section of } K^\la
\mid
\\
\text{such that} \ 
f \text{ is  holomorphic over } \Sigma\setminus A\}.\qquad\qquad
\end{multline}
As in this work the set of $A$ is fixed we do not add it to 
the notation.
Obviously, $\Fl$ is an infinite dimensional  $\C$-vector space.
Recall that in the case of half-integer $\la$ everything depends on
the theta characteristic $L$.

We call the elements of the space $\Fl$  \emph{meromorphic forms
of weight $\la$} 
(with respect to the 
theta characteristic $L$).
In local coordinates $z_i$ we can write such a form as 
$f_idz_i^\la$, with $f_i$ being a local holomorphic, resp. meromorphic function.

\subsection{Associative Multiplication}
The natural map of the 
locally free sheaves of rang one  
\begin{equation}
\K^{\lambda}\times \K^\nu\to
 \K^{\lambda}\otimes \K^\nu\cong 
\K^{\lambda+\nu},
\quad
(s,t)\mapsto s\otimes t,
\end{equation}
defines a bilinear map 
\begin{equation}\label{eq:afl}
\sbul:\Fl\times \Fl[\nu]\to \Fl[\la+\nu].
\end{equation}
With respect to local trivialisations this corresponds to the
multiplication of the local representing meromorphic functions
\begin{equation}
(s\, dz^{\la},t\, dz^{\nu} )
\mapsto s\, dz^{\la}\;\sbul\;t\, dz^{\nu}= s\cdot t\;
dz^{\la+\nu}.
\end{equation}
If there is no danger of confusion then we will mostly
use the same symbol for the section and for the local
representing function.

We set 
\begin{equation}
\mathcal{F}:=\bigoplus_{\la\in\frac 12\Z}\Fl.
\end{equation}
The following is obvious
\begin{proposition}\label{P:aass}
The vector space 
$\mathcal{F}$ is an  associative and commutative
graded (over $\frac 12\Z$)
algebra. Moreover, $\Fl[0]$ is a subalgebra.
\end{proposition}
We use also  $\A:=\Fl[0]$.
Of course, it is the  algebra of meromorphic functions on $\Sigma$
which are holomorphic outside of $A$.
The spaces $\Fl$ are modules over $\A$.
\subsection{Lie algebra structure}
Next we define a Lie  algebraic structure on the space $\mathcal{F}$.
The structure is  induced by the map 
\begin{equation}
\Fl\times \Fl[\nu]\to \Fl[\la+\nu+1],
\qquad (s,t)\mapsto [s,t],
\end{equation}
which is defined in local representatives of the sections by
\begin{equation}\label{eq:aliea}
(s\, dz^{\la},t\, dz^{\nu} )
\mapsto [s\, dz^{\la},t\, dz^{\nu}]:= \left((-\la)s\pfz{t}+\nu\, 
t\pfz{s}\right)
dz^{\la+\nu+1},
\end{equation}
and bilinearly extended to $\mathcal{F}$.
\begin{proposition}
(a)
The bilinear map $[.,.]$ defines a Lie algebra structure 
on $\mathcal{F}$.

\noindent
(b) The space  $\mathcal{F}$
with respect to $\sbul$ and 
$[.,.]$ is a Poisson algebra.
\end{proposition}
\begin{proof}
This is done by local calculations. For details 
see \cite{SchlHab}, \cite{Schlknbook}.
\end{proof}
\subsection{The vector field algebra and the Lie derivative}

\begin{proposition}
The subspace $\La:=\Fl[-1]$ is a Lie subalgebra, and the 
$\Fl$'s are Lie modules over $\La$.
\end{proposition} 
As forms of weight $-1$ are vector fields, $\La$ could also be
defined as the Lie algebra of those meromorphic vector fields 
on the Riemann surface $\Sigma$ which are holomorphic outside
of $A$.
The product \refE{aliea} gives the usual
Lie bracket of vector fields and the Lie derivative for their actions on 
forms. We get 
(again naming the  local functions
with the same symbol as the section) 
\begin{equation}\label{eq:aLbrack}
[e,f]_|(z)=[e(z)\fpz, f(z)\fpz]=
\left( e(z)\pfz f(z)- f(z)\pfz e(z)\right)\fpz \ ,
\end{equation}
\begin{equation}\label{eq:alied}
\nabla_e(f)_|(z)=L_e(g)_|=
e\ldot g_{|}=
\left( e(z)\pfz f(z)+\lambda f(z)\pfz e(z)\right)\fpz\ .
\end{equation}

\subsection{The algebra of differential operators}

In $\mathcal{F}$,
considered as Lie algebra,
$\A=\Fl[0]$ is an abelian Lie subalgebra and the vector space sum
$\Fl[0]\oplus\Fl[-1]=\A\oplus \La$  
is also a Lie subalgebra of  $\mathcal{F}$.
In an equivalent way it can also be constructed as 
semi-direct sum of $\A$ considered as abelian Lie algebra
and $\La$ operating on $\A$ by taking the derivative.
This Lie algebra is called the 
\emph{Lie algebra of differential operators of degree
$\le 1$}
and is denoted by
$\Do$.
In more direct terms $\Do=\A\oplus \La$ as vector space direct sum
and endowed with the Lie product
\begin{equation}
[(g,e),(h,f)]=(e\ldot h-f\ldot g\,,\,[e,f]).
\end{equation} 
The $\Fl$ will be Lie-modules over $\Do$.


\subsection{Superalgebra of half forms}
\label{sec:super}

Next we consider the associative product
\begin{equation}
\sbul\ \Fl[-1/2]\times \Fl[-1/2]\to \Fl[-1]=\La.
\end{equation}

Introduce the vector space and the product
\begin{equation}\label{eq:saprod}
\Sa:=\La\oplus\Fl[-1/2],\quad
[(e,\varphi),(f,\psi)]:=([e,f]+\varphi\;\sbul\;\psi,
e\ldot \varphi-f\ldot \psi).
\end{equation}
Usually we will denote
the elements of $\La$ by $e,f, \dots$, and the elements of
$\Sp$ by $\varphi,\psi,\ldots$.

Definition \refE{saprod} can be reformulated as an extension
of $[.,.]$ on  $\La$ to a ``super-bracket'' (denoted by the same
symbol) on $\Sa$ by setting
\begin{equation}\label{eq:sadef}
[e,\varphi]:=-[\varphi,e]:=e\ldot \varphi
={}_| (e\frac {d\varphi}{dz}-\frac 12\varphi\frac {de}{dz})(dz)^{-1/2}
\end{equation}
and
\begin{equation}
[\varphi,\psi]=\varphi\;\sbul\; \psi.
\end{equation}
We call the elements of $\La$ elements of even parity,
and the elements of $\Sp$ elements of odd parity.
For such elements $x$ we denote by $\bar x\in\{\bar 0,\bar 1\}$ their
parity. 

The sum \refE{saprod} can also  be described
as $\Sa=\Sa_{\bar 0}\oplus\Sa_{\bar 1}$, where $\Sa_{\bar i}$ is the
subspace of
elements of parity $\bar i$.
\begin{proposition}\label{P:knsuper}
The space $\Sa$ with the above introduced parity and product
is a Lie superalgebra.
\end{proposition}
\begin{definition}
The algebra $\Sa$ is the 
Krichever - Novikov Lie superalgebra.
\end{definition}
Before we say a few words on the proof we recall the definition
of a Lie superalgebra.
Let $\Sa$ be a vector space which is decomposed into
even and odd elements
$\Sa=\Sa_{\bar 0}\oplus\Sa_{\bar 1}$,
i.e. $\Sa$ is a $\Z/2Z$-graded vector space.
Furthermore, let $[.,.]$ be a 
  $\Z/2Z$-graded bilinear map  $\Sa\times \Sa\to \Sa$ such
that for elements $x,y$ of pure parity
\begin{equation}\label{eq:ssup}
[x,y]=-(-1)^{\bar x \bar y}[y,x].
\end{equation}
This says that
\begin{equation}\label{eq:szg}
[\Sa_{\bar 0},\Sa_{\bar 0}]\subseteq \Sa_{\bar 0},\qquad
[\Sa_{\bar 0},\Sa_{\bar 1}]\subseteq \Sa_{\bar 1},\qquad
[\Sa_{\bar 1},\Sa_{\bar 1}]\subseteq \Sa_{\bar 0},
\end{equation}
and $[x,y]$ is symmetric for $x$ and $y$ odd, otherwise
anti-symmetric.
Recall that $\Sa$ is a Lie superalgebra if in addition
the super-Jacobi identity  ($x,y,z$ of pure parity)
\begin{equation}\label{eq:jacsup}
(-1)^{\bar x \bar z}[x,[y,z]]+
(-1)^{\bar y \bar x}[y,[z,x]]+
(-1)^{\bar z \bar y}[z,[x,y]]
=0
\end{equation}
is valid.
As long as the type of the arguments is different 
from {\it (even, odd, odd)} all signs can be put to
$+1$ and we obtain the form of the usual Jacobi identity.
In the remaining  case we get
\begin{equation}
[x,[y,z]]+[y,[z,x]]-[z,[x,y]]=0.
\end{equation}
By the definitions $\Sa_0$ is a Lie algebra.

\begin{proof}(\refP{knsuper})
By \refE{saprod} Equations \refE{ssup}
and \refE{szg} are true.
If we consider \refE{jacsup} for elements of type
{\it (even, even, even)} then it reduces to the usual Jacobi identity
which
is of course true for the subalgebra of vector fields $\La$.
For {\it (even, even, odd)} it is true as $\Sp$ is a Lie-module 
over $\La$. 
For  {\it (even, odd, odd)}
we get
\begin{equation}
[e,[\varphi,\psi]]+[\varphi,[\psi,e]]-[\psi,[e,\varphi]]=
e\ldot (\varphi\;\sbul\;\psi)-(e\ldot \psi)\;\sbul\;\varphi
-(e\ldot \varphi)\;\sbul\;\psi=0,
\end{equation}
as $e$ acts as derivation on $\Sp$. 
For {\it (odd, odd, odd)}
the super-Jacobi relation writes as
\begin{equation}
[\varphi,[\psi,\chi]]+\quad\text{cyclic permutation}\qquad =\quad0.
\end{equation}
Equivalently
\begin{equation}
- (\psi\;\sbul\;\chi)\ldot \varphi+\quad\text{cyclic permutation}\qquad=\quad0.
\end{equation}

Now (again identifying local representing functions with the
element)
\begin{multline}
(\psi\;\sbul\;\chi)\ldot \varphi=
\left((\psi\cdot\chi)\cdot \varphi' - 1/2((\psi\cdot \chi)'
\varphi)\right)(dz)^{-1/2}
=
\\
\left(\psi\chi\varphi' - 1/2\,\psi'\chi\varphi
-1/2\,\psi\chi'\varphi\right)(dz)^{-1/2}.
\end{multline}
Adding up all cyclic permutations yield zero.
\end{proof}
\begin{remark}
The above introduced Lie superalgebra corresponds 
classically to the  Neveu-Schwarz superalgebra. In string theory physicists
considered also the Ramond superalgebra as string algebra
(in the two-point case). 
The elements of the Ramond superalgebra do not correspond to
sections of the dual theta characteristics. They are only 
defined on a 2-sheeted branched covering of $\Sigma$, see e.g. \cite{BMRR},
\cite{BMTW}. Hence, the elements are only multi-valued sections.
As we only consider honest sections of half-integer powers of the
canonical bundle, we do not deal with the Ramond algebra here.

The choice of the theta characteristics corresponds to choosing
a spin structure on $\Sigma$. 
For the relation of the Neveu-Schwarz superalgebra to 
the geometry of graded 
Riemann surfaces see Bryant \cite{Bryant}.
\end{remark}
\section{Almost-graded structure}
\label{sec:almgrad}
\subsection{Definition of almost-gradedness}
Recall the  classical situation. This is   
the Riemann surface  $\Pro^1(\C)=S^2$, i.e. the Riemann surface 
of genus zero,
 and the points where
poles are allowed are $\{0,\infty\}$). In this case
 the algebras introduced in the last
chapter are graded algebras. 
In the higher genus case and even in the genus zero case with more
than
two points where poles are allowed there is no non-trivial grading
anymore.
As realized by Krichever and Novikov \cite{KNFa}
there is a weaker concept, an almost-grading which
to  a large extend is a valuable replacement of a honest grading.
Such an almost-grading is induced by a splitting of the set $A$ into
two non-empty and disjoint sets $I$ and $O$. 
The (almost-)grading is fixed by exhibiting certain basis elements 
in the spaces $\Fl$ as homogeneous.
\begin{definition}\label{D:almgrad} 
Let $\La$ be a Lie or an associative algebra such that
$\La=\oplus _{n\in\Z}\La_n$ is a vector space direct sum, then
$\La$ is called an \emph{almost-graded} 
(Lie-) algebra if
\begin{enumerate}
\item[(i)] $\dim \La_n<\infty$,
\item[(ii)]
There exist constants $L_1,L_2\in\Z$ such that 
\begin{equation}
\La_n\cdot \La_m\subseteq \bigoplus_{h=n+m-L_1}^{n+m+L_2}
\La_h,\qquad \forall n,m\in\Z.
\end{equation}
\end{enumerate}
Elements in $\La_n$ are called {\it homogeneous} elements of
degree $n$, and $\La_n$ is called homogeneous subspace of degree $n$.
\end{definition}
In a similar manner almost-graded modules over almost-graded algebras
are defined. Also of course, we can extend 
in an obvious way the definition to
superalgebras, resp. even to more general algebraic structures.
This definition makes complete sense also for more general index sets
$\J$. In fact we will consider the index set
$\J=(1/2)\Z$ for our superalgebra.
Our even elements (with respect to the super-grading) will have 
integer degree, our odd elements half-integer degree.

\subsection{Separating cycle and Krichever-Novikov duality}
\label{SS:kndual}
Let $C_i$ be  positively oriented (deformed) circles around
the points $P_i$ in $I$ and $C_j^*$ positively oriented ones around
the points $Q_j$ in $O$.

A cycle $C_S$ is called a separating cycle 
if it is smooth, positively oriented of multiplicity one and if 
it separates
the in- from the out-points. It might have multiple components. 
In the following we will integrate  
meromorphic differentials on $\Sigma$ without poles in 
$\Sigma\setminus A$ over closed curves $C$. 
Hence, we might consider   $C$ and $C'$ as
equivalent if $[C]=[C']$ in  $\Ho(\Sigma\setminus A,\Z)$.
In this sense
we can write 
for every separating cycle
\begin{equation}\label{eq:cs}
[C_S]=\sum_{i=1}^K[C_i]=-\sum_{j=1}^M [C^*_j].
\end{equation}
The minus sign appears due to the opposite orientation.
Another way for giving such a $C_S$ is also via
level lines of a ``proper time
evolution'', for which I refer to Ref.~\cite{SLc}.

Given such a separating cycle $C_S$ (resp. cycle class) 
we can define a linear map
\begin{equation}
\Fl[1]\to\C,\qquad \omega\mapsto \cins \omega.
\end{equation}
As explained above the map will not depend on 
the separating line $C_S$ chosen, as
two of such will be homologous and the poles of $\omega$ are only located in
$I$ and $O$.

Consequently,
the integration of $\omega$ over $C_S$ can 
also be described
over the 
special cycles $C_i$ or equivalently over $C_j^*$.
This integration 
corresponds to calculating residues
\begin{equation}\label{eq:res}
\omega\quad\mapsto\quad
\cins \omega\quad=\quad\sum_{k=i}^K\res_{P_i}(\omega)
\quad=\quad-\sum_{l=1}^{M}\res_{Q_l}(\omega).
\end{equation}

Furthermore,
\begin{equation}\label{eq:kndual}
\Fl\times\Fl[1-\la]\to \C,\quad
(f,g)\mapsto\kndual{f}{g}:=\cins f\cdot g,
\end{equation}
gives a well-defined pairing, called the
{\it Krichever-Novikov (KN) pairing}.
\subsection{The homogeneous subspaces}
Depending on whether $\la$ is integer or half-integer, we set
$\Jl=\Z$ or $\Jl=\Z+1/2$.
For $\Fl$ we introduce for $m\in\Jl$ 
subspaces $\Fl_m$ of dimension $K$, where
$K=\#I$, 
by exhibiting  certain elements  $f_{m,p}^\la\in \Fl$,
$ p=1,\ldots, K$ which constitute a  basis of   $\Fl_m$.
Recall that the spaces $\Fl$ for $\lambda\in\Z+1/2$ depend
on the square root $L$ (the 
theta characteristic) of the canonical bundle chosen.
The elements of  $\Fl_m$ are the elements of degree $m$.

Let $I=\{P_1,P_2,\ldots, P_K\}$ 
then the basis element $f^\la_{m,p}$ of degree $m$ is of order 
\begin{equation}
\ord_{P_i}(f^\la_{m,p})
=(n+1-\la)-\delta_{i}^p 
\end{equation}
at the point $P_i\in I$,  $i=1,\ldots, K$.
The prescription at the points in $O$ is made in such a way
that the element  $f^\la_{m,p}$ is essentially uniquely given.
Essentially unique means up to multiplication with a constant%
\footnote{Strictly speaking, there are some special cases where
some constants  have to be added such  that the Krichever-Novikov
duality \refE{knd} below is valid, see \cite{SLc}.}.
After fixing as additional geometric data, a system of coordinates
$z_l$ centered at $P_l$ for $l=1,\ldots, K$ 
and requiring that 
\begin{equation}\label{eq:fnorm}
f_{n,p}^\la(z_p)=z_p^{n-\la}(1+O(z_p))(dz_p)^\la
\end{equation}
the element $f_{n,p}$ is uniquely fixed.
In fact, the  element $f_{n,p}^\la$ only depends on the first
jet of the coordinate $z_p$ \cite{SSpt}.

\begin{example}
Here we will not give the general recipe for the
prescription at the points in $O$, see 
\cite{SLc}, \cite{SDiss}, \cite{Schlknbook}.
Just to give an example which is also an
important special case, assume $O=\{Q\}$ is a one-element
set. If either the genus $g=0$, or  $g\ge 2$, $\la\ne 0,\, 1/2,\, 1$ 
 and the points in
$A$ are in generic position 
then we require     
\begin{equation}
\ord_{Q}(f^\la_{n,p})
=-K\cdot(n+1-\la)+(2\la-1)(g-1).
\end{equation}
In the other cases (e.g. for $g=1$) there are 
some modifications at the point in $O$ necessary
for finitely
many $m$.
\end{example}

The construction yields \cite{SLc}, \cite{SDiss}, \cite{Schlknbook}
\begin{theorem}\label{T:basis}
Set
\begin{equation}\label{eq:kbasis}
\mathcal{B}^\la:=\{\, 
 f_{n,p}^\la\mid n\in\Jl,\  p=1,\ldots, K\,\}.
\end{equation}
Then 
(a) $\mathcal{B}^\la$ is a basis of the vector space $\Fl$.

(b) 
The introduced basis $\mathcal{B}^\la$ of  
$\Fl[\la]$ and  $\mathcal{B}^{1-\la}$ 
of $\Fl[1-\la]$ 
are dual to each other with respect to the
Krichever-Novikov pairing \refE{kndual}, i.e.  
\begin{equation}\label{eq:knd}
\kndual{f_{n,p}^\la}{f_{-m,r}^{1-\la}}=
\delta_p^r\;\delta_n^m, \quad \forall n,m\in\Jl,\quad 
r,p=1,\ldots, K.
\end{equation}
\end{theorem}
{}From part (b) of the theorem it follows that the  
Krichever-Novikov pairing is non-degenerate. Moreover,
any element $v\in\Fl[1-\la]$ acts as linear form on $\Fl$ via
\begin{equation}
  \Fl\mapsto \C,\quad
w\mapsto \Phi_v(w):=\kndual{v}{w}.
\end{equation}
Via this pairing $\Fl[1-\la]$ can be considered as subspace of
$({\Fl})^*$. But I like to stress the fact that the identification depends
on the splitting of $A$ into $I$ and $O$ as the
KN pairing depends on it.
The full space $({\Fl})^*$ can even be described with the help
of the pairing. Consider the series
\begin{equation}\label{eq:infin}
\hat v:=\sum_{m\in\Z}\sum_{p=1}^K a_{m,p}f_{m,p}^{1-\la} 
\end{equation}
as a formal series, then
$\Phi_{\hat v}$ (as a distribution) is a well-defined  element of
${\Fl}^*$, as it will be only evaluated for finitely many
basis elements in $\Fl$.
Vice versa, every element of  ${\Fl}^*$ can be given by a suitable
$\hat v$. 
Every $\phi\in(\Fl)^*$ is uniquely given by the scalars 
$\phi(f_{m,r}^\la)$.
We set 
\begin{equation}
\hat v:=\sum_{m\in\Z}\sum_{p=1}^K \phi(f_{-m,p}^\la)\,f_{m,p}^{1-\la}. 
\end{equation}
Obviously, $\Phi_{\hat v }=\phi$.
For more information about this ``distribution
interpretation''
see \cite{SDiss}, \cite{SchlHab}.

The dual elements of $\La$ 
will be given by the formal series 
\refE{infin}  with basis elements from $\Fl[2]$, 
the quadratic differentials, and the
dual elements of $\Sp$ correspondingly from   $\Fl[3/2]$.
The spaces $\Fl[2]$ and $\Fl[3/2]$ themselves can be
considered as some kind of restricted duals.

It is quite convenient to use special notations for elements
of some important weights:
\begin{equation}
e_{n,p}:=f_{n,p}^{-1},
\quad
\varphi_{n,p}:=f_{n,p}^{-1/2},
\quad
\A_{n,p}:=f_{n,p}^{0}.
\end{equation}

\subsection{The algebras}
\begin{proposition}
\label{P:coeff}
There exist constants $R_1$ and  $R_2$ 
(depending on  the number and splitting of the points
in $A$ and of the genus $g$)  
independent of 
$n,m\in\J$ such that for the basis elements  
\begin{equation}\label{eq:coeff}
\begin{aligned}[]
f_{n,p}^\la\,\sbul\, f_{m,r}^\nu=&
\quad f_{n+m,r}^{\la+\nu}\delta_p^r
\\ &\qquad\qquad
+\sum_{h=n+m+1}^{n+m+R_1}\sum_{s=1}^Ka_{(n,p)(m,r)}^{(h,s)}f_{h,s}^{\la+\nu},
\quad a_{(n,p)(m,r)}^{(h,s)}\in\C,
\\ \hbox{}
\\
[f_{n,p}^\la, f_{m,r}^\nu]=&
\quad
(-\la m+\nu n)\,f_{n+m,r}^{\la+\nu+1}\delta_p^r
\\
&\qquad\qquad\qquad +
\sum_{h=n+m+1}^{n+m+R_2}\sum_{s=1}^Kb_{(n,p)(m,r)}^{(h,s)}f_{h,s}^{\la+\nu+1},
\quad b_{(n,p)(m,r)}^{(h,s)}\in\C.
\end{aligned}
\end{equation}
\end{proposition}
\begin{proof}
For the elements on the l.h.s. of \refE{coeff} we can estimate
the maximal pole orders at the points in $I$ and $O$. Using the KN duality
and the prescribed orders of the basis elements we obtain
by considering possible pole orders at the points in $I$
the lower
bound of the degree, and by considering the pole orders
at $O$ the upper bounds $R_1$ and $R_2$ for the degree on the r.h.s.. 
The degree $n+m$ part follows from local calculations at 
the points in $I$.
See \cite{SLc}, \cite{SDiss}, \cite{Schlknbook} for more details.
\end{proof}

As a direct consequence we obtain
\begin{theorem}\label{T:almgrad}
The algebras  $\La$ and  $\Sa$ are almost-graded Lie , resp.
Lie superalgebras.
The almost-grading depends on the splitting of the set $A$ into $I$
and $O$. 
More precisely,
\begin{equation}
\Fl=\bigoplus_{m\in\Jl}\Fl_m,\qquad\text{with}\quad \dim \Fl_m=K.
\end{equation}
and there exist $R_1,R_2$ (independent of $n$ and
$m$)
such that
\begin{equation*}
[\La_n,\La_m] \subseteq \bigoplus_{h=n+m}^{n+m+R_1}\La_h\ ,
 \qquad
[\Sa_n, \Sa_m] \subseteq \bigoplus_{h=n+m}^{n+m+R_2}\Sa_h.
\end{equation*}
\end{theorem}
The constants $R_i$ depend on the genus of the Riemann surface and  the
number of points in $I$ and $O$. In fact they
can be explicitly calculated (if needed).

Also from \refE{coeff} we can directly conclude
\begin{proposition}\label{P:leading}
For all $m,n\in \Jl$ and $r,p=1,\ldots,K$ we have
\begin{equation}\label{eq:coeffalm}
\begin{aligned}[]
[e_{n,p},e_{m,r}]&=(m-n)\cdot e_{n+m,r}\,\delta_r^p+\text{h.d.t.}
\\
e_{n,p}\ldot \varphi_{m,r}&=
(m-\frac{n}{2})\cdot \varphi_{n+m,r}\,\delta_r^p+\text{h.d.t.}
\\
\varphi_{n,p}\;\sbul\;\varphi_{m,r}&=e_{n+m,r}\,\delta_r^p+\text{h.d.t.}
\end{aligned}
\end{equation}
Here $\text{h.d.t.}$ denote linear combinations of 
basis elements of degree between
$n+m+1$ and $n+m+R_2$, with the $R_2$ from \refT{almgrad}.
\end{proposition}
See \refS{classext} for an example in 
the classical case  (by ignoring the central
extension appearing there for the moment).

\begin{remark}
Note that in certain literature for the classical situation some other
normalisation of the structure equation of the superalgebra
was given. Instead of the product of the -1/2 forms twice the
product was used (corresponding to  the anti-commutator). These results
can also be obtained by setting
$\varphi_{n,p}=\sqrt{2}f_{n,p}^{-1/2}$.
The last line of \refE{coeffalm} will then start with  
$2e_{n+m,p}\delta_p^r$. This has also consequences for
the structure equation  of the central extension
\refE{classcent}.
\end{remark}

On the basis of the almost-grading we obtain a triangular
decomposition of the algebras
\begin{equation} 
\La=\La_{[+]}\oplus\La_{[0]}\oplus\La_{[-]},\qquad\quad
\Sa=\Sa_{[+]}\oplus\Sa_{[0]}\oplus\Sa_{[-]}
\end{equation}
where e.g.
\begin{equation}
\Sa_{[+]}:=\bigoplus_{m>0}\Sa_m,\quad
\Sa_{[0]}=\bigoplus_{m=-R_2}^{m=0}\Sa_m,\quad
\Sa_{[-]}:=\bigoplus_{m<-R_2}\Sa_m.
\end{equation}
By the almost-graded structure the $[+]$ and
$[-]$ subspaces are indeed 
(infinite dimensional) subalgebras. The $[0]$ spaces in general
are not subalgebras.

\begin{remark}\label{R:filt}
In case that $O$ has more than one point
there are certain choices, e.g. numbering of the points in
$O$, different rules, etc. involved. 
Hence, if the choices are made differently the subspaces $\Fl_n$ 
might depend on them, and consequently also the almost-grading.
But as it is shown in the above quoted works  the 
induced filtration
\begin{equation}
\begin{gathered}
\Fl_{(n)}:=\bigoplus_{m\ge n} \Fl_m,
\\ ....\quad
\supseteq\quad  \Fl_{(n-1)}\quad
\supseteq \quad\Fl_{(n)}\quad
\supseteq \quad \Fl_{(n+1)}\quad ....
\end{gathered}
\end{equation}
has an intrinsic meaning
given  by
\begin{equation}
\Fl_{(n)}:=\{\ f\in\Fl\mid \ord_{P_i}(f)\ge n-\la,\forall i=1,\ldots, K\
\}\; .
\end{equation}
Hence it is independent of  these choices.
\end{remark}
\begin{remark}\label{jordan}
Leidwanger and Morier-Genoux introduced in \cite{LeiMor} also a
Jordan superalgebra based on the Krichever-Novikov objects, i.e.
\begin{equation}
\mathcal{J}:=\Fl[0]\oplus\Fl[-1/2]=
\mathcal{J}_{\bar 0}\oplus
\mathcal{J}_{\bar 1}.
\end{equation}
Recall that $\Fl[0]$ is the associative algebra of 
meromorphic functions.
The (Jordan) product is defined via the algebra structure 
introduced in \refS{algebra} for the spaces $\Fl$ by
\begin{equation}
\begin{aligned}[]
f\circ g&:=f\;\sbul\; g\quad \in\Fl[0],
\\
f\circ\varphi&:=f\;\sbul\; \varphi \quad \in\Fl[-1/2]
\\
\varphi\circ\psi&:=[\varphi,\psi]\quad \in\Fl[0].
\end{aligned}
\end{equation}
By rescaling the second definition with the factor 1/2 one obtains
a Lie antialgebra. See \cite{LeiMor} for more details
and additional results on representations.

Here I only want to add the following. Using the results presented
in this section one easily sees that with respect to the
almost-grading introduced (depending on a splitting 
$A=I\cup O$) the Jordan superalgebra  becomes indeed an
almost-graded algebra
\begin{equation}
\mathcal{J}=\bigoplus_{m\in 1/2\Z}\mathcal{J}_m.
\end{equation}
Hence, it makes sense to call it a Jordan superalgebra of
KN type.
Calculated for the introduced basis elements 
we get (using \refP{coeff})
\begin{equation}\label{eq:coeffjord}
\begin{aligned}[]
A_{n,p}\circ A_{m,r}&=A_{n+m,r}\,\delta_r^p+\text{h.d.t.}
\\
A_{n,p}\circ \varphi_{m,r}&=
\varphi_{n+m,r}\,\delta_r^p+\text{h.d.t.}
\\
\varphi_{n,p}\circ\varphi_{m,r}&=\frac 12(m-n)\A_{n+m,r}\,\delta_r^p+\text{h.d.t.}
\end{aligned}
\end{equation}
\end{remark}

\section{Central extensions}
\label{sec:centvec}

In this section we recall the results which are needed 
about central extensions of the 
vector field algebra in the following discussion of central
extensions of the Lie superalgebras.
More details can be found in \cite{SLc}, \cite{SDiss}, \cite{Scocyc}.

A central extension of a Lie algebra $W$ is defined on the
vector space direct sum 
\newline
$\widehat{W}=\C\oplus W$.
If we denote  $\hat x:=(0,x)$ and $t:=(1,0)$ then its Lie structure is
given by 
\begin{equation}\label{eq:cext}
[\hat x, \hat y]=\widehat{[x,y]}+\Phi(x,y)\cdot t,
\quad [t,\widehat{W}]=0,
\quad x,y\in W.
\end{equation}
$\widehat{W}$ will be a Lie algebra, e.g. fulfill the 
Jacobi identity, if and only if $\Phi$ is antisymmetric and fulfills
the Lie algebra 2-cocycle condition
\begin{equation}
0=d_2\Phi(x,y,z):=
\Phi([x,y],z)+
\Phi([y,z],x)+
\Phi([z,x],y).
\end{equation}
There is the notion of equivalence of central extensions. It turns out
that two central extensions are equivalent if and only if the difference of 
their defining
2-cocycles $\Phi$ and $\Phi'$ is a coboundary, i.e. there exists 
a $\phi:W\to\C$ such that 
\begin{equation}
\Phi(x,y)-\Phi'(x,y)=d_1\phi(x,y)=\phi([x,y]).
\end{equation}
In this way the second Lie algebra cohomology $\Ho^2(W,\C)$ 
of $W$ with 
values in the trivial module $\C$ classifies equivalence classes
of central extensions. The class $[0]$ corresponds to the
trivial (i.e. split) central extension. 
Hence, to construct central extensions of our Lie algebras we have to 
find such Lie algebra 2-cocycles. 

We want to generalize the cocycle which defines in the classical
case the Virasoro algebra to higher genus and the multi-point 
situation. We have to geometrize the cocycle.
Before we can give this 
geometric description, we have to introduce the notion of a  
{\it projective connection}
\begin{definition}
Let $\ (U_\al,z_\al)_{\al\in J}\ $ be a covering of the Riemann surface
by holomorphic coordinates, with transition functions
$z_\be=f_{\be\al}(z_\al)$.
A system of local holomorphic functions 
$\ R=(R_\al(z_\al))\ $ 
is called a holomorphic {\it projective 
connection} if it transforms as
\begin{equation}\label{eq:pc}
R_\be(z_\be)\cdot (f_{\beta,\alpha}')^2=R_\al(z_\al)+S(f_{\beta,\alpha}),
\qquad\text{with}\quad
S(h)=\frac {h'''}{h'}-\frac 32\left(\frac {h''}{h'}\right)^2,
\end{equation}
the Schwartzian derivative.
Here ${}'$ means  differentiation with respect to
the coordinate $z_\al$.
\end{definition}
It is a classical result 
\cite{HawSchiff}, \cite{Gun} that 
every Riemann surface admits a holomorphic projective connection $R$.
{}From the definition it follows that the difference of two
projective
connections is a quadratic differential. In fact starting from one
projective projection we will obtain all of them by adding quadratic
differentials to it. 

\medskip

Given a smooth differentiable curve $C$
(not necessarily connected) and
a fixed holomorphic projective connection $R$,
the following  defines
 for the vector field algebra a two-cocycle 
\begin{equation}\label{eq:vecg}
\Phi_{C,R}(e,f):=\cintl{C} \left(\frac 12(e'''f-ef''')
-R\cdot(e'f-ef')\right)dz\ .
\end{equation}
Only by the term involving the 
 projective connection it will be
a well-defined differential, i.e. independent of the 
chosen coordinates.
It is shown in Ref.~\cite{SDiss} that it 
is a cocycle. Another choice of a projective connection
will result in a cohomologous one, see also \refE{proj}

In contrast to the classical situation, 
for the higher genus and/or multi-point situation
there are many essentially
different closed curves and also many non-equivalent central
extensions defined by the integration.

But we should take into account that we want to extend the
almost-grading from our algebras to the centrally extended ones.
This means we take $\deg\hat x:=\deg x$ and assign a degree 
$deg(t)$ to the central element $t$, and  obtain
an almost-grading.

This is possible if and only if our defining cocycle $\psi$ is 
``local'' in the following sense (the name was introduced 
in the two point case by
Krichever and Novikov in Ref.~\cite{KNFa}). There exists 
$M_1,M_2\in\Z$ such that
\begin{equation}
\forall n,m:\quad\psi(W_n,W_m)\ne 0\ \implies\  
M_1\le n+m\le M_2.
\end{equation}
Here $W$ stands for any  of our algebras (including the supercase
discussed below).
Very important, ``local'' is defined in terms of the
grading, and the grading itself depends on the splitting
$A=I\cup O$. Hence what is ``local''  depends on the splitting too.

We will call a cocycle {\it bounded} (from above) if there exists
$M\in\Z$ such that
\begin{equation}
\forall n,m:\quad\psi(W_n,W_m)\ne 0\ \implies\  
n+m\le M.
\end{equation}
Similarly bounded from below can be defined. Locality means bounded
from above and below.

Given a cocycle class we call it bounded (resp. local) if and only if
it contains a representing cocycle which is bounded (resp. local).
Not all cocycles in a bounded class have to be bounded.
If we choose as integration path a separating cocycle $C_S$,
or one of the $C_i$  then
the above introduced geometric cocycles \refE{vecg} are local,
resp. bounded.
Recall that in this case integration can  be done by calculating
residues at the in-points or  at the out-points.
All these cocycles are cohomologically nontrivial.
The following theorem concerns the opposite direction.

\begin{theorem}\label{T:vclass}
\cite{Scocyc} 
Let $\La$ be the Krichever--Novikov vector field algebra.
\newline
(a) The space of bounded 
cohomology classes is $K$-dimensional ($K=\#I$).
A basis is given by setting  the integration path in
\refE{vecg}  to $C_i$, $i=1,\ldots, K$ the little (deformed) circles
around the points $P_i\in I$.
\newline
(b) The space of local cohomology classes is one-dimensional.
A generator is given by integrating \refE{vecg} over a
separating cocycle $C_S$.
\newline
(c) Up to equivalence and rescaling there is only one one-dimensional
central extension of the vector field algebra $\La$ which 
allows  an extension of the almost-grading.
\end{theorem}
\section{Central extensions - the supercase}
In this section we 
consider central extensions of our Lie superalgebra $\Sa$.
Such a central extension is 
given by a bilinear map 
\begin{equation}
c:\Sa\times\Sa\to \C
\end{equation}
via an expression
completely analogous  to \refE{cext}.
Additional conditions for $c$ follow from the fact that
the resulting extension should be again a superalgebra.
This implies that for the homogeneous elements $x,y,z\in\Sa$
($\Sa$ might be an arbitrary Lie superalgebra) we have
\begin{equation}\label{eq:csymm}
c(x,y)=-(-1)^{\bar x \bar y}c(x,y).
\end{equation}
The bilinear map $c$  will be symmetric if $x$ and $y$ are odd, otherwise
it will be antisymmetric.
The super-cocycle condition reads in complete analogy with 
the super-Jacobi relation as
\begin{equation}\label{eq:scocyc}
(-1)^{\bar x \bar z}c(x,[y,z])+
(-1)^{\bar y \bar x}c(y,[z,x])+
(-1)^{\bar z \bar y}c(z,[x,y])\ = \ 0.
\end{equation}
As we will need it anyway, I will write it out
for the different type of arguments.
For {\it (even,even,even)}, {\it (even,even, odd)}, and 
{\it (odd,odd,odd)}
it will be of the ``usual form'' of the cocycle condition
\begin{equation}\label{eq:scocyn}
c(x,[y,z])+
c(y,[z,x])+
c(z,[x,y])\ = \ 0  .
\end{equation}
For {\it (even,odd,odd)} we obtain
\begin{equation}\label{eq:scocys}
c(x,[y,z])+
c(y,[z,x])-
c(z,[x,y])\ = \ 0 .
\end{equation} 
Now we have to decide which parity our central element should have.
In our context the natural choice is that the central element
should be even, as we want to extend the central extension of
the vector field algebra to the superalgebra. 
This implies that our bilinear form $c$ has to be an even form.
Consequently,  
\begin{equation}
c(x,y)=c(y,x)=0, \quad \text{for \ } \bar x=0,\bar y=1.
\end{equation}
In this case from the super-cocycle conditions 
only \refE{scocys} for the {\it (even,odd,odd)} and 
\refE{scocyn} for the {\it (even,even,even)}
case will give relations which are not nontrivially zero.
In \refS{odd} we will consider the case that the central element is of
odd parity.

Given a linear form $k:\Sa\to\C$ we assign to it
\begin{equation}
\delta_1k(x,y)=k([x,y]).
\end{equation}
As in the classical case 
$\delta_1k$ will be a super-cocycle.
A super-cocycle will be a coboundary 
if and only if there exists a linear form $k:\Sa\to\C$ such that
$c=\delta_1k$. As $k$ is a linear form it can be written as
$k=k_{\bar 0}\oplus k_{\bar 1}$ where 
$k_{\bar 0}:\Sa_{\bar 0}\to \C$ and
$k_{\bar 1}:\Sa_{\bar 1}\to \C$.
Again we have the two cases of the parity of the central element.
Let $c$ be  a coboundary $\delta_1k$. If the central element is even then
$c$  will also be a coboundary of a $k$ with
$k_{\bar 1}=0$. In other words $k$ is even. 
In the odd case we have  $k_{\bar 0}=0$ and $k$ is odd.

After fixing a parity of the central element we consider 
the quotient spaces
\begin{align}
\Ho^2_{\bar 0}(\Sa,\C)&:=\{\text{even cocycles}\}/
\{\text{even coboundaries}\},
\\
\Ho^2_{\bar 1}(\Sa,\C)&:=\{\text{odd cocycles}\}/
\{\text{odd coboundaries}\}.
\end{align}
These cohomology spaces classify central extensions of $\Sa$ with
even (resp. odd) central elements up to equivalence.
Equivalence is  defined as in the non-super setting.

\medskip
For the rest of this section our algebra $\Sa$ will the Lie superalgebra 
introduced in \refS{super}. Moreover, for the moment we concentrate
on the case of an even central element $t$.
Recall our convention to denote vector fields by $e,f,g,...$ and
-1/2-forms by $\varphi,\psi,\chi,..$.
{}From the discussion above we know
\begin{equation}
c(e,\varphi)=0,\quad e\in\La,\ \varphi\in\Sp.
\end{equation}
The super-cocycle conditions for the even elements is just the
cocycle condition for the Lie subalgebra $\La$. The only other
nonvanishing super-cocycle condition is for the 
{\it (even,odd,odd)} elements and reads as
\begin{equation}\label{eq:cyccon}
c(e,[\varphi,\psi])-c(\varphi,e\ldot \psi)-c(\psi,e\ldot \varphi)
=0.
\end{equation}
Here the definition of the product $[e,\psi]:=e\ldot\psi$
was used to rewrite  \refE{scocys}. 

In particular, if we have a cocycle $c$ for the algebra $\Sa$ we obtain
by restriction a cocycle for the algebra $\La$.
For the mixing term we know that $c(e,\psi)=0$.
A naive try to put just anything for $c(\varphi,\psi)$ will not work
as \refE{cyccon} relates the restriction of the cocycle on $\La$ with
its values on $\Sp$.

\medskip
\begin{proposition}
Let $C$ be any closed (differentiable) curve on $\Sigma$ not meeting
the
points in $A$, and let $R$ be any (holomorphic) projective connection,
then the bilinear extension of 
\begin{equation}
\label{eq:sacoc}
\begin{aligned}
\Phi_{C,R}(e,f)&:=\cintl{C} \left(\frac 12(e'''f-ef''')
-R\cdot(e'f-ef')\right)dz
\\
\Phi_{C,R}(\varphi,\psi)
&:=
-\cintl{C} \left(\varphi''\cdot \psi+\varphi\cdot \psi''
-R\cdot\varphi\cdot \psi\right)dz
\\
\Phi_{C,R}(e,\varphi)&:=0
\end{aligned}
\end{equation}
gives a Lie superalgebra cocycle for $\Sa$, hence defines a central
extension of $\Sa$
\end{proposition}
A similar formula was given  by
Bryant in \cite{Bryant}. By adding the projective connection in
the second part of 
\refE{sacoc} he corrected some formula appearing in 
\cite{BMRR}. He only considered the two-point case and only the
integration over a separating cycle.
See also \cite{Kreusch} for the multi-point case, where  still only
the integration over a separating cycle is considered.
\begin{proof}
First one has to show that the integrands are 
well-defined differentials. It is exactly this point for which 
$R$ had to be introduced. This was done for the vector field case in
\cite{SDiss}, \cite{SLc}. The case for the second part of \refE{sacoc} 
is completely 
analogous, and follows from straight-forward calculations.

 Next the super-Jacobi identities have to be verified.
Again the first one \refE{scocyn} was shown in the latter references.
For the other one \refE{cyccon} we write out the three integrands 
and sum them up (before we
integrate over $C$). By direct calculations we obtain that the term 
coming with the projective connection will identically vanish. Hence
the rest will be a well-defined meromorphic differential, It will not 
necessarily vanish
identically. We only claim that the sum will vanish after integration
over an arbitrary  closed curve.  Recall that the curve integral over an
exact meromorphic differential $\omega$, i.e. a differential
 which is exact, i.e. can be written
locally as $\omega=df$ with $f$ a
meromorphic function on $\Sigma$, will vanish.
In a first step we calculate 
\begin{align*}
e'''f&=\frac 12(e'''f-ef''')+1/2((ef)''-3(e'f'))'  
\\
-2(\varphi'\psi')&=(\varphi''\psi+\varphi\psi'')
-\left((\varphi\psi)'\right)'.
\end{align*}
Hence, we can replace the corresponding integrands
in the cocycle expressions by  integrands given by 
the left hand side.
The total integrand of \refE{cyccon} can 
now be written as
$Bdz$ with 
\begin{equation}
B: =e'''(\varphi\psi)-2\varphi'(e\psi'-1/2e'\psi)'
-2\psi'(e\varphi'-1/2e'\varphi)' 
\end{equation}
which calculates to 
\begin{equation}
B=(e''(\varphi\psi)-2(e\varphi'\psi'))'.
\end{equation}
Consequently $Bdz$ integrated over any closed curve $C$ will
vanish. This shows that the cocycle condition
\refE{cyccon} is true.
\end{proof}

How will the central extension depend on $C$ and $R$?
Obviously,
two cycles lying in the same homology class class of $\Sigma\setminus
A$ will define the same cocycle.  

\begin{proposition}
If $R$ and $R'$ are two  projective connections then
$\Phi_{C,R}$ and $\Phi_{C,R'}$ are cohomologous.
Hence the cohomology class will not depend on the choice of $R$.
\end{proposition}
\begin{proof}
The difference of two projective connections is a quadratic
differential, $\Omega=R'-R$. 
We calculate
\begin{equation}\label{eq:proj}
\begin{gathered}
\Phi_{C,R'}(e,f)-\Phi_{C,R}(e,f)=\cintl{C} \Omega\cdot(ef'-e'f)dz
=\cintl{C} \Omega\cdot [e,f],
\\
\Phi_{C,R'}(\varphi,\psi)-\Phi_{C,R}(\varphi,\psi)
=\cintl{C} \Omega\cdot(\varphi\cdot\psi)dz
=\cintl{C} \Omega\cdot [\varphi,\psi].
\end{gathered}
\end{equation}
If we fix the quadratic differential
$\Omega$ then the map
\begin{equation}
\kappa_C:\La\to\C,\qquad  e\mapsto \cintl{C} \Omega\cdot e
\end{equation}
is a linear map. 
We extend this map by zero on $\Sp$ and obtain using
\refE{proj} that
\begin{equation} 
\Phi_{C,R'}-\Phi_{C,R}=\delta_1\kappa_C.
\end{equation}
Hence both cocycles are cohomologous.
\end{proof}

As in the pure vector field case 
for the non-classical situation,
there will be many inequivalent
central extensions given by different
cycle classes as integration paths. 
Recall that classical means $g=0$ and $N=2$.

We will need
 the special integration paths $C_i$, ($C_j^*$), the circles around the
points
$P_i\in I$ ($Q_j\in O$), 
introduced in \refS{almgrad} and $C_S$ a separating 
cycle. 
Recall from \refE{cs}
\begin{equation}\label{eq:cs1}
[C_S]=\sum_{i=1}^K[C_i]=-\sum_{j=1}^M [C^*_j],
\end{equation}
as  homology classes.
\begin{proposition}\label{P:bounded}
(a) The cocycles $\Phi_{C_i,R}$ are bounded (from above) by zero.

(b) The cocycle $\Phi_{C_S,R}$ obtained by integrating over a separating
cocycle is  a local cocycle.
\end{proposition}
\begin{proof} 
First, the cocycles evaluated for the  vector field subalgebra are
bounded  resp. local, as  shown in \cite{SLc}, 
\cite[Thm. 4.2]{Scocyc}.
The same argument works for the other part of the cocycle.
Just to give the principle idea:
We consider the integrand for  pairs of elements
$(\varphi_{m,p},\psi_{n,r})$. If $m+n>0$ it will not have 
residues at the points $P_i$. Hence the integration around $C_i$ will
yield 0. This shows (a) and the fact that $\Phi_{C_S,R}$ is bounded from
above by zero. Integration over $C_S$ can alternatively be done also
by summation of integration over the right hand side of expression
\refE{cs1}. By the definition of the homogeneous elements 
there is a bound $S$ independent of $n,m$ such that 
the integrand for pairs of elements for which the sum
of their degrees $<S$, do not 
have poles at the points $Q_j\in O$. Hence the integral will vanish
too.
This says the cocycle $\Phi_{C_S,R}$ is bounded from below,
hence local.
\end{proof}

\medskip
The question is, will the opposite be also true, meaning that
every local or every bounded cocycle will be equivalent to
those cocycles defined above, resp. to a certain linear combination 
of them?
As in the vector field algebra case, it will turn out
that with respect to a fixed almost-grading the
non-trivial almost-graded central extension (with even central element) 
will be essentially unique, hence given by
the a scalar times $\Phi_{C_S,R}$. 
Also we will make a corresponding statement
about bounded cocycles.

The following is the crucial technical result
\begin{proposition}\label{P:vanish}
Let $c$ be a cocycle for the superalgebra $\Sa$ which is bounded 
from above, and which vanishes on the vector field subalgebra $\La$,
then $c$ vanishes in total. 
In other words, every bounded cocycle is uniquely given by its restriction to
the vector field subalgebra.
\end{proposition}
Before we prove this proposition in \refS{proof} 
we formulate the main theorem of this article.
\begin{theorem}\label{T:main}
Given the Lie superalgebra $\Sa$ 
of Krichever-Novikov type with
its induced almost-grading given by the splitting of $A$ into
$I$ and $O$. Then:

(a) The space of bounded cohomology classes  
 has dimension $K=\#I$. A basis is given by 
the classes of cocycles 
\refE{sacoc} 
integrating over the cycles $C_i$, $i=1,\ldots, K$.

(b) The space of local cohomology classes is one-dimensional.
A generator is given by the class of \refE{sacoc}
integrating over a  separating cocycle $C_S$.

(c) Up to equivalence and rescaling there is only one non-trivial
almost-graded central extension of the Lie superalgebra
extending the almost-graded structure on $\Sa$.
\end{theorem}

\begin{proof}
Let $c$ be a bounded cocycle for $\Sa$. After restriction to
$\La\times\La$ we obtain a bounded cocycle for $\La$. By
\refT{vclass} it is cohomologous to 
a standard cocycle with a suitable projective connection $R$
\begin{equation}
\Phi:=\sum_{i=1}^K a_i\Phi_{C_i,R},\quad a_i\in\C.
\end{equation}
Let $\kappa:\La\to\C$ be the linear form giving
the coboundary, i.e. $c_{_|\La}-\Phi=\delta_1(\kappa)$.
We extend $\kappa$ by zero for the elements of
$\Fl[-1/2]$. 
Then $\Phi':=c-\delta_1(\kappa)$ is cohomologous to
$c$ (as cocycle for $\Sa$). Moreover,
$(c-\delta_1(\kappa))_{|\La}=\Phi$.
Next we extend $\Phi$ via \refE{sacoc} to $\Sa$ (i.e. taking
the same linear combination of integration cycles and the same projective 
connection) and obtain another cocycle for $\Sa$, still denoted by
$\Phi$. 
By construction the difference $\Phi'-\Phi$ is zero on $\La$. Hence by 
\refP{vanish} it vanishes on $\Sa$. But $\Phi$ is exactly
of the form claimed in (a). We get exactly $K$ linearly
independent cocycle classes, as they are linearly independent as
cocycles for the subalgebra $\La$.
Claim (b) follows in a completely analogous manner, now applied
to local cocycles. Claim (c) is a direct consequence.
\end{proof}

I like to stress the fact, that it will not be the case that there is
only one central extension of the superalgebra $\Sa$. 
Only if we fix an almost-grading for $\Sa$, which means that we
fix a splitting
of $A$ into $I$ and $O$, there will be a unique central extension
allowing us to extend the almost-grading.
For another essentially different splitting (meaning that it
is not only the changing of the role of $I$ and $O$) 
splitting the almost-grading
will be different and we will obtain a different central extension of
the algebra.
In fact, if the genus $g$ of the Riemann surface is larger than zero,
there will be non-equivalent central extensions which are not
associated to any almost-grading (neither coming from a local
cocycle, nor from a bounded cocycle).

\subsection{Proof of \refP{vanish}}\label{sec:proof}
We start with a bounded cocycle $c$ for $\Sa$ which vanishes on the subalgebra 
$\La$ and consider 
the cocycle condition \refE{cyccon}.
It remains
\begin{equation}\label{eq:con}
c(\varphi,e\ldot \psi)+c(\psi,e\ldot \varphi)
=0, \quad \forall e\in\La,\varphi,\psi\in\Sp.
\end{equation}
Our goal is to show that it is identically zero.

For a pair of homogeneous basis elements $(f_{m,p},g_{n,r})$
of any combination of types we call $l=n+m$ the {\it level} of the
pair. We evaluate the cocycle $c$ at pairs of level $l$, i.e.
$c(f_{m,p},g_{n-l,r})$. We call these cocycle values, values of level
$l$.
We apply the technique  developed in
\cite{Scocyc}. We  will consider  cocycle values
$c(f_{m,p},g_{n,r})$ on pairs of level $l=n+m$ and will make
descending
induction over the level. By the boundedness from above, the
cocycle values will vanish at all pairs of sufficiently high
level. It will turn out that everything will be fixed by the
values of the cocycle at level zero. Finally, we will show
that the cocycle  $c$ also vanishes at level zero.
Hence the claim of \refP{vanish}.

For a cocycle $c$ evaluated for pairs of elements of level $l$
we will use the symbol $\equiv$ to denote that the expressions are
the same on both sides of an equation involving cocycle values
up to values of $c$  at
higher level. This has to be understood in the following strong
sense:
\begin{equation}
\sum \alpha^{(n,p,r)}c(f_{n,p},g_{l-n,r})\equiv 0,\qquad
 \alpha^{(n,p,r)}\in\C
\end{equation}
denotes a congruence modulo a linear combination of values of $c$
at pairs of basis elements of level $l'>l$. The coefficients of
that linear combination, as well as the  $\alpha^{(n,p,r)}$,
depend only on the structure of the  Lie algebra $\Sa$ and do not
depend on $c$.
We will also use the same symbol $\equiv$ for equalities in  $\Sa$
which are true modulo terms of higher degree compared to the terms
under consideration.

We consider the triple of basis elements $(\varphi_{m,r},\varphi_{n,s},e_{k,p})$
for level $l=n+m+k$.
Recall \refE{coeffalm}
\begin{equation}
[e_{k,p},\varphi_{n,s}]\equiv(n-\frac k2)\,\delta_p^s\,\varphi_{n+k,p}.
\end{equation}
Hence
\begin{equation}
c(\varphi_{m,r},[e_{k,p},\varphi_{n,s}])\equiv
c(\varphi_{m,r},(n-\frac k2)\,\delta_p^s\,\varphi_{n+k,p})
=(n-\frac k2)\,\delta_p^s\,c(\varphi_{m,r},\varphi_{n+k,p}).
\end{equation}
If we use this we obtain from \refE{con}
\begin{equation}\label{eq:con1}
(n-\frac k2)\,\delta_p^s \,c(\varphi_{m,r},\varphi_{n+k,p})+
(m-\frac k2)\,\delta_p^r\, c(\varphi_{n,s},\varphi_{m+k,p})\equiv 0.
\end{equation}
We set $k=0$ then  (now the level is $n+m$)
\begin{equation}\label{eq:con2}
n\,\delta_p^s\, c(\varphi_{m,r},\varphi_{n,p})+
m\,\delta_p^r\, c(\varphi_{n,s},\varphi_{m,p})\equiv 0.
\end{equation}
If $p=s$ but $p\ne r$ then we obtain 
\begin{equation}
n\cdot c(\varphi_{m,r},\varphi_{n,p})
\equiv 0.
\end{equation}
As $n\in\Z+\frac 12$, and hence $n\ne 0$ we have
\begin{equation}\label{eq:con4}
  c(\varphi_{m,r},\varphi_{n,p})
\equiv 0, \quad r\ne p.
\end{equation}
Next we consider $r=p=s$ and \refE{con2} yields
\begin{equation}
n\cdot  c(\varphi_{m,s},\varphi_{n,s})+
m\cdot  c(\varphi_{n,s},\varphi_{m,s})\equiv 0.
\end{equation}
As $c$ is symmetric on $\Sp$ we get
\begin{equation}\label{eq:con5}
(n+m)\cdot  c(\varphi_{m,s},\varphi_{n,s})\equiv 0.
\end{equation}
This shows that, as long as the level is different
from zero, the cocycle is given via universal cocycle
values of higher level.
By assumption our cocycle $c$ is bounded from above. Hence there
exists a level $R$ such that for all levels $>R$ the cocycle
values will vanish. We get by induction from 
\refE{con4} and \refE{con5} that the cocycle values will be zero
for all levels $>0$. Next we will show that it will vanish also
at level zero.
We go back to \refE{con1} (for $s=p=r$) and plug in $n=m$ and
$k=-2n\in\Z$ and obtain
\begin{equation}
4n\cdot c(\varphi_{n,s},\varphi_{-n,s})
\equiv 0, \quad \forall n\in \Z+\frac 12.
\end{equation}
Hence also at level zero everything will be expressed by cocycle 
values of higher level and consequently will be equal to zero.
Continuing with \refE{con5} we see that also at level $<0$ we get
that the cocycle will vanish. Hence the claim
\qed.
\subsection{The case of an odd central element}
\label{sec:odd}
$ $

In this section we will consider the case where the central element
has odd parity. 
We will show
\begin{theorem}\label{T:odd}
Every bounded cocycle yielding a central extension with odd
central element is a coboundary.
\end{theorem}
Hence,
\begin{corollary}
There are  no non-trivial central extensions of the Lie superalgebra
$\Sa$  with odd central element coming from a bounded cocycle.
In other words all such central extensions will split.
\end{corollary}
\begin{proof}(\refT{odd})
In the odd case 
only the cocycle relations \refE{scocyn} for the
{\it (even,even,odd)} and {\it (odd,odd,odd)} combinations will 
be non-trivial. 
We first make a cohomologous change by defining 
recursively a map
\begin{equation}
\Phi:\Sp\to\C,
\end{equation}
which will be extended by zero on $\La$.
We consider $c(e_{0,p},\varphi_{k,r})$. It is of level $k$. By the
boundedness there exists an $R$ such that for  
$k>R$  all its values will vanish. 
Recall e.g. \refE{coeffalm}
\begin{equation}
e_{0,p}\ldot \varphi_{k,p}=k\cdot\varphi_{k,p}+y_{k,p},
\end{equation}
with $y_{k,p}$ a finite sum of elements of degree $\ge k+1$.
Set 
\begin{equation}
\Phi(\varphi_{k,p}):=0,\quad k>R,\  p=1,\ldots,K
\end{equation}
and then recursively for $k=R, R-1, .....$
\begin{equation}
\Phi(\varphi_{k,p}):=\frac 1k\left(c(e_{0,p},\varphi_{k,p})
-\Phi(y_{k,p})\right),\quad p=1,\ldots,K.
\end{equation}
For the cohomologous cocycle
$c'=c-\delta_1\Phi$
we calculate for $p=1,\ldots, K$
\begin{equation}
c'(e_{0,p},\varphi_{m,p})=
c(e_{0,p},\varphi_{m,p})-\Phi(e_{0,p}\ldot \varphi_{k,p})
=c(e_{0,p},\varphi_{m,p})-k\Phi(\varphi_{k,p})-\Phi(y_{k,p})
=0.
\end{equation}
{\bf Claim:} The cocycle $c'$ vanishes identically.
For simplicity we will drop the ${}'$.
We consider the  cocycle relation for {\it (even,even,odd)}
for the elements $e_{n,p}, e_{m,r}, \varphi_{k,s}$.
With the same technique used in the last section we obtain
\begin{equation}
(k-\frac m2)c(e_{n,p},\varphi_{k+m,r})\delta_r^s
-
(k-\frac n2)c(e_{m,r},\varphi_{k+n,s})\delta_p^s
-(m-n)
c(e_{m+n,r},\varphi_{k,s})\delta_p^r\equiv 0.
\end{equation}
For $n=0$
this specializes to 
\begin{equation}\label{eq:con6}
(k-\frac m2)c(e_{0,p},\varphi_{k+m,r})\delta_r^s
-k\cdot c(e_{m,r},\varphi_{k,s})\delta_p^s
-m\cdot 
c(e_{m,r},\varphi_{k,s})\delta_p^r\equiv 0.
\end{equation}
If we set $s=p\ne r$ then 
\begin{equation}
-k\cdot c(e_{m,r},\varphi_{k,s})\equiv 0.
\end{equation}
As $k$ is half-integer  we get
$c(e_{m,r},\varphi_{k,s})\equiv 0$.
For  $r=s=p$ in \refE{con6} we obtain 
\begin{equation}
(k-\frac m2)c(e_{0,p},\varphi_{m,p})
-
(k+m)c(e_{m,p},\varphi_{k,p})
\equiv 0.
\end{equation}
As $c(e_{0,p},\varphi_{m,p})=0$ and $k+m\ne 0$ we get
$c(e_{m,p},\varphi_{k,p})\equiv 0$ for all $m$ and $k$.
Hence all values are determined by  values of higher
level. By the boundedness it will be zero at all level.
\end{proof}

\subsection{Some special examples}
In this section I like to give reference to some special
examples.

\subsubsection{Higher genus, two points}
In this case there is only one almost-grading for $\Sa$ 
because there is only  one possible splitting.
In this case the separating cycle $C_S$ coincides with the
 cycle $C_1$. In particular, integration over $C_1$ already gives
a local cocycle. But this does not mean that every bounded 
cocycle will be local, it only means that the class of bounded cocycles 
coincides with the class of local cocycles.
This case with integration over $C_S$ was considered by 
Bryant \cite{Bryant}. But he does not prove uniqueness.
Also it has to be repeated that for higher genus, there are
other non-equivalent cocycles obtained by integration by other
non-trivial cycle classes on $\Sigma$.
See also Zachos \cite{Zac} for $g=1$ and two points.

\subsubsection{Genus $g=0$, two points}\label{sec:classext}
This is the classical situation. 
By some isomorphism of  $\Sigma=S^2$, the Riemann sphere,
we can assume that $I=\{0\}$ and $O=\{\infty\}$ with respect 
to the quasi-global coordinate $z$. As projective connection 
$R=0$ will do.
In this situation our algebras are honestly graded and the
elements can be given like follows
\begin{equation}
e_n=z^{n+1}\frac {d}{dz},\quad n\in\Z,\qquad
\varphi_m=z^{m+1/2}(dz)^{-1/2},\quad m\in\Z+1/2.
\end{equation}
By calculating the cocycle values we obtain the well-known
expressions
\begin{equation}\label{eq:classcent}
\begin{aligned}[]
[e_n,e_m]
&=(m-n)e_{m+n}+\frac 1{12}(n^3-n)\,\delta_n^{-m}\, t,
\\
[e_n,\varphi_m]&=(m-\frac n2)\;\varphi_{m+n},
\\
[\varphi_n,\varphi_m]&=e_{n+m}-\frac 16\,(n^2-\frac 14)\,\delta_n^{-m}\, t.
\end{aligned}
\end{equation}
In fact, 
all higher order terms in the calculations above 
are now exact not only up to higher order.
This means that there is no reference to boundedness needed and
the statements are true for all cocycles.  This is the same as for
the vector field algebra. 
Note also, that the subspace
\begin{equation}
\langle e_{-1}, e_{0},  e_{-1},
\varphi_{-1/2}, \varphi_{1/2}
\rangle
\end{equation}
is a finite-dimensional sub Lie superalgebra. It consists of the global holomorphic
sections of $\Fl[-1]$ and $\Fl[-1/2]$. Restricted to this
subalgebra the cocycle vanishes.

\subsubsection{Genus $g=0$, more than two points}

Here our algebra will not be graded anymore, but only
almost-graded. Different splittings give different
separating cycles and non-equivalent central extensions.
The $N=3$ situation was studied by Kreusch \cite{Kreusch}
and the central extension was calculated
independently of this work.
The case $N=3$ is somehow special. If we fix
the 3 points then we have 3 essentially different splitting into
$I$ and $O$.  
Hence we also have  3 non-equivalent different central
extensions.
If we fix one splitting then 
by a  biholomorphic mapping of $S^2$ 
any other splitting  can be mapped to this one. 
Such a mapping 
induces an automorphism of the 
algebra $\Sa$ (and of
course
also of $\La$). Hence, the obtained central extensions will be isomorphic
(but not equivalent).



\begin{thebibliography}{10}



\bibitem{BMRR}
Bonora, L., Martellini, M., 
Rinaldi, M.,
Russo, J.,
\emph{Neveu-Schwarz- and Ramond-type Superalgebras on genus
$g$ Riemann surfaces},
Phys. Lett. B, 206(3) (1988), 444--450.


\bibitem{BLMR}
Bonora, L., Lugo, A., Matone, M., 
Russo, J.,
\emph{A global operator formalism on higher genus Riemann surfaces:
$b-c$ systems},
Comm. Math. Phys. 123 (1989), 329--352.

\bibitem{BMTW}
Bonora, L.,  Matone, M., Toppan, F., 
Wu, K.,
\emph{Real weight $b-c$ systems and conformal field theories in
higher genus}, 
Nuclear Phys. B 334 (1990), no. 3, 717--744. 

\bibitem{Bryant}
Bryant, P.,
\emph{Graded Riemann surfaces and Krichever-Novikov algebras},
Lett. Math. Phys., {19} (1990), 97--108.


\bibitem{FiaSchl1}
Fialowski, A., Schlichenmaier, M.,
\emph{Global deformations of the Witt algebra of Krichever-Novikov type},
Comm. Contemp. Math. {5} (6) (2003), 921--946.

\bibitem{FiaSchlaff}

Fialowski, A., Schlichenmaier, M.,
\emph{Global geometric deformations of current algebras
as Krichever-Novikov type algebras},
Comm. Math. Phys. 260 (2005), 579 --612.

\bibitem{GSW1}
Green, M.B., Schwarz, J.H., Witten, E.,
\emph{Superstring theory}, Vol.1,
Cambridge University Press, 1987


\bibitem{GR}
Guieu, L., Roger, C.:
\emph{L'alg\`ebre et le groupe de Virasoro.}
Les publications CRM, Montreal 2007.


\bibitem{Gun}
Gunning, R.C., 
\emph{Lectures on Riemann surfaces},
Princeton Math. Notees, N.J. 1966.

\bibitem{HawSchiff}
Hawley, N.S., Schiffer, M.,
\emph{Half-order differentials on Riemann surfaces},
Acta Math. 115 (1966), 199--236.


\bibitem{KacB}
Kac, V.G., \emph{Infinite dimensional {L}ie algebras}. Cambridge Univ. Press,
  Cambridge, 1990.

\bibitem{Kreusch}
Kreusch, M.,
\emph{Extensions of superalgebras of Krichever-Novikov type},
arXiv:1204.4338v2.


\bibitem{KNFa}
Krichever, I.M., Novikov, S.P., 
\emph{Algebras of {V}irasoro type, {R}iemann
  surfaces and structures of the theory of solitons}. Funktional Anal. i.
  Prilozhen. {21}, No.2 (1987), 46-63.

\bibitem{KNFb}
Krichever, I.M., Novikov, S.P., 
\emph{Virasoro type algebras, Riemann surfaces and strings in Minkowski
space}. Funktional Anal. i.
  Prilozhen. {21}, No.4 (1987), 47-61.

\bibitem{KNFc}
Krichever, I.M., Novikov, S.P., 
\emph{Algebras of Virasoro type, energy-momentum tensors and
decompositions
 of operators on Riemann surfaces}.
 Funktional Anal. i.
  Prilozhen. {23}, No.1 (1989), 46-63.


\bibitem{LecOvs}
Lecomte, P., Ovsienko, V.,
\emph{Alternated Hochschild cohomology},
arXiv:1012.3885.


\bibitem{LeiMor}
Leidwanger, S., Morier-Genoud, S.,
\emph{Superalgebras associated to Riemann surfaces: 
Jordan algebras of Krichever-Novikov type},
Int. Math. Res. Notices, doi:10.1093/imrn/rnr196



\bibitem{LeiMor1}
Leidwanger, S., Morier-Genoud, S.,
\emph{Universal enveloping algebras of Lie antialgebras},
Algebras and Representation Theory 15(1) (2012), 1--27


\bibitem{Ovs1}
Ovsienko, V.,
\emph{Lie antialgebras: pr\'emices},
J. of Algebra 325 (1) (2011), 216--247


\bibitem{SLa}
Schlichenmaier, M.,
{\it Krichever-Novikov algebras for more than two points}.
Lett.  Math. Phys. {19} (1990), 151-165.

\bibitem{SLb}
Schlichenmaier, M.,
{\it Krichever-Novikov algebras for more than two points:
explicit generators}. Lett. Math. Phys.
{19} (1990), 327-336.

\bibitem{SLc}
Schlichenmaier, M.,
{\it Central extensions and semi-infinite wedge representations of
Krichever-Novikov algebras for more than two points}.
Letters in Mathematical Physics {20} (1990), 33-46.


\bibitem{SDiss}
Schlichenmaier, M.,\emph{Verallgemeinerte {K}richever - {N}ovikov {A}lgebren
  und deren {D}arstellungen}. Ph.D. thesis, Universit{\"{a}}t Mannheim, 1990.

\bibitem{SchlDeg}
Schlichenmaier, M.,
{\it
Degenerations of generalized
Krichever-Novikov algebras on tori},
Journal of Mathematical Physics {34}(1993), 3809-3824.


\bibitem{SchlHab}
Schlichenmaier, M.,
{\it
Zwei Anwendungen algebraisch-geometrischer Methoden in der
theoretischen Physik: Berezin-Toeplitz-Quantisierung und
globale Algebren der zweidimensionalen konformen Feldtheorie''},
Habilitation Thesis, University of Mannheim, June 1996.


%
\bibitem{Scocyc}
Schlichenmaier, M., \emph{Local cocycles and central extensions for
multi-point algebras of Krichever-Novikov type}. J. Reine und
Angewandte Mathematik {559} (2003), 53--94.


\bibitem{Schlknbook}
Schlichenmaier, M.,
\emph{Krichever-Novikov algebras. Theory and Applications},
forthcoming.


\bibitem{SSpt}
Schlichenmaier, M., Sheinman, O.K.,
 \emph{{W}ess-{Z}umino-{W}itten-{N}ovikov theory,
  {K}nizhnik-{Z}amolodchikov equations, and {K}richever-{N}ovikov algebras,
  {I}.}, Russian Math. Surv. (Uspekhi Math. Nauk.) \textbf{54} (1999),
  213--250, math.QA/9812083.


\bibitem{SSpt2}
Schlichenmaier, M., Sheinman, O.K.,
{\it
Knizhnik-Zamolodchikov equations for positive genus and
Krichever-Novikov algebras},
Russian Math. Surv. {59} (2004), No. 4, 737--770,

\bibitem{Zac}
Zachos, C.K.,
{\it
Fermionic center in the superconformal algebra on the 
supertorus},
Nucl. Phys. B (Proc. Suppl.) 11 (1989), 414--424.




\end{thebibliography}
\end{document}